\documentclass[oneside,english,11pt]{amsart}

\usepackage[hmargin=2.5cm,vmargin=2cm]{geometry}        

\DeclareMathAlphabet\mathbfcal{OMS}{cmsy}{b}{n}

\newcommand{\noop}[1]{}
\usepackage[T1]{fontenc}
\usepackage[utf8]{inputenc} 	
\usepackage[english]{babel}
\usepackage{indentfirst}
\usepackage{lmodern}
\usepackage{url}
\usepackage[breaklinks]{hyperref}
\usepackage{amsmath}
\usepackage{amsfonts}
\usepackage{amssymb,graphics}
\usepackage{verbatim}
\usepackage{eurosym} 						
\usepackage{tikz}
\usetikzlibrary{matrix}
\theoremstyle{definition}
\newtheorem{definition}{Definition}[section]

\theoremstyle{plain}
\newtheorem{theoreme}[definition]{Theorem}
\newtheorem{fact}[definition]{Fact}
\newtheorem*{theoreme*}{Theorem}
\newtheorem{corollaire}[definition]{Corollary}

\newtheorem{proposition}[definition]{Proposition}
\newtheorem{lemme}[definition]{Lemma}
\newtheorem{Aff}{Claim}

\theoremstyle{remark}
\newtheorem{remarque}[definition]{Remark}

\newtheoremstyle{case}{}{}{}{}{}{}{ }{}
\theoremstyle{case}
\newtheorem*{case}{{\bf Case}}

\numberwithin{equation}{section}

\title{NIP, and $\textrm{NTP}_2$ division rings of prime characteristic}
\date{}							
\author{Cédric Milliet}
\address{Département de mathématiques, Université de Mons,\newline
 Le Pentagone, 20, Place du Parc,\newline
 B-7000 Mons, Belgique}

\email[]{cedric.milliet@gmail.com}
\subjclass[2010]{14R99, 14A22, 12E15, 03C45, 03C60}
\keywords{Division ring, model theory, independence property, tree property of the second kind}
\thanks{Many thanks to Franziska Jahnke for answering questions and pointing at \cite[Theorem 5.2]{JK2015a} and \cite[Theorem 3.10]{JK2015b}, and to the referee for his or her patient readings and suggestions.}

\newcommand{\C}{{\rm C}}
\newcommand{\Z}{{\rm Z}}
\newcommand{\V}{{\rm V}}
\newcommand{\I}{{\rm I}}

\newcommand{\R}{{ D}}

\newcommand{\N}{{\rm N}}
\newcommand{\ima}{\mathrm{Im}}
\newcommand{\id}{{\rm id}}

\newcommand{\FR}{{\rm Fix}(\sigma)}
\newcommand{\Rs}{{\R[\sigma]}}

\newcommand{\NIP}{{\rm NIP}}
\newcommand{\NTP}{{${\rm NTP}_2$}}

\newcommand{\Rn}{\R[\sigma,n]}

\newcommand{\cl}{{\rm cl}}

\begin{document}

\begin{abstract}Combining a characterisation by Bélair, Kaplan, Scanlon and Wagner of certain NIP valued fields of characteristic $p$ with Dickson's construction of cyclic algebras, we provide examples of noncommutative \NIP{} division ring of  characteristic~$p$ and show that an \NIP{} division ring of characteristic $p$ has finite dimension over its centre, in the spirit of Kaplan and Scanlon's proof that infinite NIP fields have no Artin-Schreier extension. The result extends to \NTP{} division rings of characteristic $p$, using Chernikov, Kaplan and Simon's \cite{CKS2015}. We also highlight consequences of our proofs that concern \NIP{} or simple difference fields.\end{abstract}

\maketitle

\section{Introduction}

Macintyre proved any $\omega$-stable field to be either finite or algebraically closed \cite[Theorem 1]{Macintyre1971}. This was generalised by Cherlin and Shelah to superstable fields \cite[Theorem 1]{CS1980}. It follows that a superstable division ring is a field \cite[Theorem p. 99]{Cherlin1977}. It was observed around 1991 that a division ring interpretable in a bounded PAC field $K$ (\emph{e.g.} a pseudo-finite field) is definably isomorphic to a finite field extension of $K$, and in particular commutative \cite[Theorem 9.1]{Hrushovski2002}. Later on, it was shown in \cite[Theorem 5.1]{PSW1998} that any supersimple division ring is a field. In another direction, Pillay proved that an infinite field  definable in an \mbox{o-minimal} structure is either real closed or algebraically closed \cite[Theorem 3.9]{Pillay1988}, and such a field has characteristic 0. It is shown in \cite[Theorem 1.1]{OPP1996} that a division ring definable in an o-minimal expansion of a real closed field $R$ is definably isomorphic to either $R$, $R\sqrt{-1}$ or the quaternions over~$R$. This was generalised to division rings definable in any o-minimal structure in \cite[Theorem 4.1]{PS1999}. A context that includes (almost) all the abovementioned structures is the one of superrosy structure, endowed with an abstract notion of ordinal valued rank on definable sets, preserved under definable bijections and satisfying Lascar's inequalities. It is shown in \cite[Theorem 2.9]{HP2018} that a superrosy division ring has finite dimension over its centre.

More can be said in characteristic $p$, even in the absence of a well-behaved global rank. It is known that a stable division ring of characteristic $p$ is a finite dimensional algebra over its centre \cite[Theorem 2.1]{Milliet2011}. Whereas the only known stable division rings are commutative fields (the conjecture that stable fields are separably closed implies that stable division rings of characteristic~$p$ are commutative), Hamilton's Quaternions over the real or 2-adic numbers are noncommutative examples of \NIP{} division rings of characteristic~$0$. The paper exhibits noncommutative examples of \NIP{} division rings of characteristic~$p$ (Theorem~\ref{ex1}), provides another simple proof that a stable division ring of characteristic $p$ has finite dimension over its centre (Fact~\ref{stable}) and shows that the same conclusion holds for an \NIP{} division ring of characteristic $p$ (Theorem~\ref{gegen}).

The proof of Theorem~\ref{gegen} closely follows ideas of Kaplan and Scanlon's \cite[Theorem 4.3]{KSW2011} stating that an infinite \NIP{} field of characteristic $p$ does not have any proper Artin-Schreier extension. Our guiding line is the reminiscence from superstability that a well-behaved definable group morphism with a ``small'' kernel should have a ``large'' image. To achieve that, a Zariski dimension theory is developed in \cite{Milliet2018} for subgroups of $(\R^n,+)$ defined over a division ring~$\R$ by linear equations involving a ring morphism $\sigma$. This dimension on a class of quantifier-free definable sets replaces the absence of a well-behaved model-theoretic rank. Sets of dimension zero include finite sets, but also $\FR$ and right affine spaces of finite $\FR$-dimension.

Eventually, Using Chernikov, Kaplan and Simon's descending chain condition for \NTP{} groups \cite[Theorem 2.4]{CKS2015}, as well as the same authors' generalisation of the definable case of Wagner's \cite[Theorem 3.2]{KSW2011}, stating that an \NTP{} field has only finitely many proper Artin-Schreier extensions \cite[Theorem 3.1]{CKS2015}, we extend Theorem~\ref{gegen} to the case of \NTP{} division rings of characteristic~$p$ (Theorem~\ref{T:N}), which has the unexpected consequence that the centre of an infinite \NTP{} division ring is infinite. Examples of strictly \NTP{} fields of characteristic $0$ are given in \cite{CKS2015}, \cite{CH2014} and \cite{Montenegro2017}, and corresponding examples in characteristic $p$ seem to be unknown.

We begin by recalling the definition of an \NIP{} structure. Given a natural number $k\in\mathbf N$ and a structure $(M,L)$, an $L$-formula $\phi(x,\bar y)$ has the \emph{$k$-independence property} if there are tuples $(a_1,\dots,a_k)$ and $\left(\bar b_J\colon J\subset\{1,\dots,k\}\right)$ in $M$ such that for any $i<k+1$ and $J\subset\{1,\dots,k\}$, \[\left(M\models\phi(a_i,\bar b_J)\right)\iff i\in J.\]

\begin{definition}[Shelah]A structure $(M,L)$ is \NIP{} (a shorthand for ``not the independence property'') if for every $L$-formula $\phi(x,\bar y)$, there is a natural number $k\in\mathbf N$ such that $\phi$ does not have the $k$-independence property.\end{definition}

Groups which are uniformly definable in an \NIP{} structure satisfy the following Noetherian like condition (see \cite[Lemme 1.3]{Poizat1987} or \cite[Theorem 1.0.5]{Wagner1997} for a proof), which seems to have appeared following \cite[p. 270]{BS1976}.

\begin{fact}[\NIP{} descending chain condition]\label{BS}In an \NIP{} group, to any formula $\phi(x,\bar y)$ is associated a natural number $n\in\mathbf N$ such that the intersection of any {\bf finite} family $\{G_i\colon i<k\}$ of subgroups defined respectively by the formulas $\{\phi(x,\bar a_i)\colon i<k\}$ be the intersection of at most $n$ among them.\end{fact}

\section{Examples of {NIP} division rings of prime characteristic}

\begin{theoreme}\label{ex1}There are noncommutative \NIP{} division rings of every characteristic.\end{theoreme}

\begin{proof}We recall Dickson's construction of \emph{cyclic algebras} as explicated in \cite[p. 229]{Lam1991}. Let $K/F$ be a Galois extension with cyclic Galois group ${\rm Gal}(K/F)$ generated by an automorphism $\sigma$ of order $s=\dim_F K$. Fixing a nonzero element $\alpha\in F$ and a symbol $x$, we let \[D=K\cdot 1\oplus K\cdot x\oplus\cdots\oplus K\cdot x^{s-1},\] and multiply elements in $D$ by using the distributive law, and the two rules \[x^s=\alpha,\quad x\cdot a=\sigma(a)x\quad \text{(for any $a\in K$)}.\] As $F\subset \Z(D)$, the ring $D$ is an $F$-algebra, of dimension $s^2$. This algebra is denoted by $(K/F,\sigma,\alpha)$, and is called the \emph{cyclic algebra associated with $(K/F,\sigma)$ and $\alpha\in F\setminus\{0\}$}. Let $\N_{K/F}\colon K^\times\rightarrow F^\times$ denote the \emph{norm map} of the extension $K/F$ defined by \[\N_{K/F}(a)=\prod_{\tau\in{\rm Gal}(K/F)}\tau(a).\] In general, $D=(K/F,\sigma,\alpha)$ need not be a division algebra, but one has:

\begin{fact}[\protect{\cite[Corollary 14.8]{Lam1991}}]\label{LamD}Suppose $s$ is a prime number. Then $D=(K/F,\sigma,\alpha)$ is a division algebra if and only if $\alpha\notin \N_{K/F}(K^\times)$.\end{fact}

Now let $p$ be a prime number different from $2$, let $\Gamma$ be the ordered additive subgroup $\left\langle 1/p^i\colon i\in\mathbf N\right\rangle$ of $\mathbf R$, and consider an \NIP{} perfect field $k$ of characteristic $p$ having an element $\alpha\in k$ with no square root in $k$. For instance, using the following Fact~\ref{Bel} from \cite[Theorem 5.9]{KSW2011} (and from \cite[Corollaire 7.5]{Bel1999}), one may consider for $k$ the field $\mathbf F_p^{alg}((x^\Gamma))$ of formal Hahn series $\displaystyle\sum a_\gamma x^\gamma$ having a well ordered support in $\Gamma$ and coefficients $a_\gamma\in\mathbf F^{alg}_p$ and take $\alpha=x$.

\begin{fact}[B\'elair, Kaplan, Scanlon and Wagner]\label{Bel}Let $(F,v)$ be an algebraically maximal valued field of characteristic $p$ whose residue field $k$ is perfect. Then $(F,v)$ is \NIP{} if and only if $k$ is \NIP{} and infinite and $\Gamma$ is $p$-divisible.\end{fact}

With its natural valuation $v$ mapping a series to the minimum of it support, the valued field $(k,v)=\left(\mathbf F_p^{alg}((x^\Gamma)),v\right)$ is \emph{maximal}, \emph{i.e.} has no proper valued field extension having both same residue field and same valuation group (see \cite{Krull1932} or \cite[Exercise 3.5.6]{EP2005}). Its residue field $\mathbf F_p^{alg}$ is algebraically closed, hence \NIP{}. Its valuation group $\Gamma$ is $p$-divisible, so the pure field $k$ is \NIP{} by Fact~\ref{Bel}, and $\alpha=x$ does not have a square root in $k$. Note that $k$ is perfect since a series $\displaystyle\sum a_\gamma x^\gamma$ has a $p$th-root $\displaystyle\sum a_\gamma^{1/p} x^{\gamma/p}$. Let us consider the field $F=k((t^\Gamma))$. Again, by Fact~\ref{Bel}, the pure field $F$ is \NIP{}. The extension $F(\sqrt t)/F$ has a cyclic Galois group generated by the automorphism $\sigma$ switching $\sqrt t $ and $-\sqrt t$. The cyclic algebra $D=(F(\sqrt t)/F,\sigma,\alpha)$ is an $F$-algebra of centre $F$ and dimension 4, definable in~$F$ (as $\sigma$ is definable in $F(\sqrt t)$), so the ring $D$ does not have the independence property. Since the norm map $\N_{F(\sqrt t)/F}$ is defined by \[\N_{F(\sqrt t)/F}(a+b\sqrt t)=(a+b\sqrt t)(a-b\sqrt t)=a^2-b^2t,\] we claim that $\alpha$ does not belong to $\N_{F(\sqrt t)/F}(F(\sqrt t)^\times)$. Assume for a contradiction that $a^2-b^2t=\alpha$ holds for some $(a,b)$ in $F$. Let $a_\gamma t^\gamma$ and $b_\delta t^\delta$ be the monomials of smallest valuation appearing in $a$ and $b$ respectively (where $a_\gamma$ and $b_\delta$ are elements of~$k$, possibly zero if $a$ or $b$ are zero). The monomials of smallest valuation appearing in $a^2$ and $b^2t$ are $a_\gamma^{2}t^{2\gamma}$ and $b_\delta^{2}t^{2\delta+1}$ respectively. Since $2\Gamma$ and $2\Gamma+1$ are disjoint, one has either $a_\gamma^{2}t^{2\gamma}=\alpha$, or $-b_\delta^{2}t^{2\delta+1}=\alpha$. The first case leads to $a_\gamma^2=\alpha$, a contradiction since $\alpha$ was chosen with no square root in $k$, and the second case to $2\delta+1=0$, a contradiction as well. We conclude by Fact~\ref{LamD} that $D$ is an \NIP{} division ring.

Note that what is needed for the present purpose is:\begin{itemize}
\item that $F$ be \NIP{}, so that  $D=(K/F,\sigma,\alpha)$ be \NIP{} as well,
\item that $\alpha$ belong to $F^\times\setminus\N_{K/F}(K^\times)$, so that $D$ be a division ring.\end{itemize}
If $p=2$, we let $\Gamma=\langle 1/2^i\colon i\in\mathbf N\rangle$ and chose similarly a perfect \NIP{} field $k$ of characteristic $2$ having an element $\alpha\in k$ with no third-root in $k$, and having a primitive third-root $\omega$ of $1$. For instance, we may take $k=\mathbf F_2^{alg}((x^\Gamma))$ and $\alpha=x$. We then consider the \NIP{} field $F=k((t^\Gamma))$ and do a similar construction as above with the cyclic $F$-algebra $D=(F(\sqrt[3] t)/F,\sigma,\alpha)$ of dimension $9$ where $\sigma\in{\rm Gal}(F(\sqrt[3]t/F))$ is the automorphism mapping $a+b\sqrt[3]t+c\sqrt[3]t^2$ to $a+\omega b\sqrt[3] t+\omega^2c\sqrt[3] t^2$. Using the identity $1+\omega+\omega^2=0$, one shows that the norm map $\N_{F(\sqrt[3] t)/F}$ is defined by \begin{align*}\N_{F(\sqrt[3] t)/F}(a+b\sqrt[3] t+c\sqrt[3] t^2)&= (a+b\sqrt[3] t+c\sqrt[3] t^2)(a+\omega b\sqrt[3] t+\omega^2c\sqrt[3] t^2)(a+\omega^2b\sqrt[3] t+\omega c\sqrt[3] t^2)\\
&= a^3+b^3t+c^3t^2+3abc(\omega^2t+\omega t) \\
&= a^3+b^3t+c^3t^2+abct.\end{align*}
We claim that $\alpha$ does not belong to $\N_{F(\sqrt[3]t)/F}(F(\sqrt[3]t)^\times)$. Assume for a contradiction that $a^3+b^3t+c^3t^2+abct=\alpha$ holds for some $(a,b,c)$ in $F$, and let $a_\gamma t^\gamma$, $b_\delta t^\delta$ and $c_\varepsilon t^\varepsilon$ be the monomials of smallest valuation appearing in $a$, $b$ and $c$ respectively (where $a_\gamma,b_\delta$ and $c_\varepsilon$ are elements of~$k$, possibly zero). The monomials of smallest valuation appearing in $a^3$, $b^3t$, $c^3t^2$ and $abct$ are $a_\gamma^3t^{3\gamma}$, $b_\delta^3t^{3\delta+1}$, $c_\varepsilon^3t^{3\varepsilon+2}$ and $a_\gamma b_\delta c_\varepsilon t^{\gamma+\delta+\varepsilon+1}$ respectively. Since $3\Gamma$, $3\Gamma+1$ and $3\Gamma +2$ are pairwise disjoint, and since $\gamma+\delta+\varepsilon+1$ is the arithmetic mean of $3\gamma$, $3\delta+1$ and $3\varepsilon +2$, one has $$\min\{3\gamma,3\delta+1,3\varepsilon +2\}<\gamma+\delta+\varepsilon+1.$$ It follows that either $a_\gamma^3t^{3\gamma}$, or $b_\delta^3t^{3\delta+1}$ or $c_\varepsilon^3t^{3\varepsilon+2}$ equals $\alpha$, but either case leads to a contradiction.\end{proof}
The pure division rings constructed above are not stable since their centres are Henselian (see \cite[Corollary 18.4.2]{Erfrat2006}) and have a nontrivial definable valuation (see for example \cite[Theorem 5.2]{JK2015a} or \cite[Theorem 3.10]{JK2015b}).

\section{Preliminaries on NIP division rings of prime characteristic}

\subsection{NIP Fields}It is believed that an \NIP{} field is either finite, separably closed, real closed or admits a nontrivial henselian valuation (this conjecture is attributed in \cite{HHJ2019} to S. Shelah). A characterisation of the subclass of dp-minimal fields is given in \cite{Johnson2015}, which also confirms Shelah's conjecture for the particular case of dp-minimal fields. According to \cite{Johnson2015}, the main Theorem ``almost says that all infinite dp-minimal fields are elementary equivalent to ones of the form $k((t^\Gamma))$ where $k$ is $\mathbf F_p^{alg}$ or a characteristic zero local field, and $\Gamma$ satisfies some divisibility conditions. The one exceptional case is the mixed characteristic case, which includes fields such as the spherical completion of $\mathbf Z_p^{un}(p^{1/p^\infty})$.'' In addition to the Baldwin-Saxl chain condition~\ref{BS} for intersections of uniformly definable subgroups, we shall only use the following result from \cite[Theorem 4.3]{KSW2011}. Let us recall that if $F$ is a field of characteristic~$p$, a proper field extension $K/F$ is called \emph{Artin-Schreier} if $K=F(a)$ where $a$ is a root of $x^p-x+b$ for some $b\in F$.

\begin{fact}[Kaplan and Scanlon]\label{KSW}An infinite \NIP{} field has no Artin-Schreier extension.\end{fact}

The proof of Fact~\ref{KSW} strongly relies on the fact that a connected algebraic subgroup of $(K,+)^n$ of Zariski dimension $1$ is isomorphic to $(K,+)$ when $K$ is a perfect field. As an immediate Corollary of Fact~\ref{KSW}, using the result of Duret \cite[Théorème 6.4]{Duret1979} on weakly algebraically closed non separably closed fields (see \cite[Corollary 4.5]{KSW2011}),

\begin{fact}[Kaplan and Scanlon]\label{KSW2}An infinite \NIP{} field of characteristic $p$ contains ${\mathbf F}_p^{alg}$.\end{fact}

\subsection{Metro equation in NIP division rings of prime characteristic}Let us first remark that in a division ring having finite dimension over its centre and characteristic $0$, the equation \[xy-yx=1\] has no solution. For putting $\gamma_x(y)=xy-yx$, a simple induction shows that $\gamma_x(y)=1$ implies $\gamma_x(y^n)=ny^{n-1}$ for every $n\in\mathbf N$, forcing the chain $\ker\gamma_x\subset\cdots\subset\ker\gamma_x^n$ of vector-spaces to be properly ascending and contradicting the finiteness of the dimension. The same conclusion fails in characteristic~$p$, and P. Cohn provides the following general condition in \cite[p. 68]{Cohn1973} (also reported by Lam \cite[p. 239]{Lam2003}) for an arbitrary division ring $\R$.

\begin{fact}[Cohn]\label{LamB}Let $a\in \R$ be algebraic over $\Z(\R)$. Then $ax-xa=1$ has a solution if and only if $a$ is not separable over $\Z(\R)$.\end{fact}

The equation $ax-xa=1$ arose in a conversation between P. Cohn and S. Amitsur on the Paris Metro on the 28th of June 1972 according to \cite[p. 418]{Cohn1995}, and is referred to as the \emph{metro-equation} in \cite{Cohn1973}. Our first goal is to show that the metro equation has no solution in an \NIP{} division ring of characteristic $p$. For that purpose, we recall Herstein's Lemma.

\begin{fact}[Herstein \protect{\cite[Lemma 3.1.1]{Her1996}}]\label{H}Let $a\in \R^\times\setminus \Z(\R)$ have finite multiplicative order. There is $b\in \R^\times$ and a natural number $n\in\mathbf N$ such that \[b^{-1}ab=a^n\neq a.\]\end{fact}

It is pointed out in \cite[Exercise 16.17]{Lam2003} that Fact~\ref{H} holds in every characteristic. In characte\-ristic $p$, the element $b$ in Fact~\ref{H} has infinite order, for otherwise $a$ and $b$ would generate a finite (noncommutative) integral domain, contradicting Wedderburn's Little Theorem. It follows that in an infinite division ring, any element $a$ has an infinite centraliser, which we write $\C(a)$. For if $a$ has infinite order, then $\C(a)$ contains the infinite cyclic group $\left\langle a\right\rangle$, whereas if $a$ has finite order~$q$, Herstein's Lemma yields a $b$ with $b^{-1}ab=a^n$ where $a$ and $a^n$ have same order~$q$, so that $n$ and $q$ are coprime. Writing $\varphi$ for Euler's totient function, Euler's Theorem provides that $\C(a)$ contains the infinite $\left\langle b^{\varphi(q)}\right\rangle$ (see also \cite[Theorem 13.10]{Lam1991}).  One may use instead Brauer's \cite[Corollary 3.3.9]{Cohn1995} which implies that any algebraic element over $\Z(\R)$ has a ``large'' centraliser.

\begin{fact}[Brauer \cite{Brauer1932}]\label{F:B}For any $a\in\R$, one has $[\R:\C(a)]_{\rm left}=[\Z(\R)(a):\Z(\R)]$.\end{fact}

By symmetry, Brauer's result implies that for any $a$, the division ring $\R$ has equal right and left $\C(a)$-dimension, which we may write $[\R:\C(a)]$ without ambiguity.

\begin{theoreme}\label{ZZZ}The centre of an infinite \NIP{} division ring is infinite.\end{theoreme}

\begin{proof}Let $\R$ be an infinite \NIP{} division ring of characteristic $p$. If all elements have finite order, by Fact~\ref{H}, the ring $\R$ is commutative, so we may assume that there is some $c\in\R^\times$ having infinite order. The field $\Z\left(\C(c)\right)$ is infinite. By Fact~\ref{KSW2}, it contains a copy of $\mathbf F_p^{alg}$. We have shown that any infinite \NIP{} division ring contains a copy of $\mathbf F_p^{alg}$. We claim that this copy is unique and lies in the centre of $\Z(\R)$. For that purpose, since any centraliser $\C(a)$ contains a copy $F_a$ of $\mathbf F_p^{alg}$, it suffices to fix a natural number $n\in\mathbf N$ and show that any two roots $(\omega_1,\omega_2)$ of $x^{p^n}-x$ commute. This will provide that $F_a=F_b$ for any $(a,b)$ in $\R$. Note that one has $[\R:\C(\omega_i)]< p^n$ by Fact~\ref{F:B}. It follows that the division ring $\C(\omega_1)\cap\C(\omega_2)$ is infinite, and \NIP{}, so contains a copy $F_n$ of $\mathbf F_{p^n}$. But one has $F_n(\omega_1)=F_n=F_n(\omega_2)$ since the polynomial $x^{p^n}-x$ has already $p^n$ roots in $F_n$, so $\omega_1$ and $\omega_2$ commute. This shows that $\Z(\R)$ contains $\mathbf F_p^{alg}$.
\end{proof}

\begin{corollaire}[metro equation]\label{metro}An \NIP{} division ring of characteristic $p$ satisfies \mbox{$xy-yx\neq1$}.\end{corollaire}

\begin{proof}We assume that the division ring is infinite, and first claim $\C(a^p-a)\subset\C(a)$ for every $a$. The field $\Z\left(\C(a^p-a)\right)(a)$ is an Artin-Schreier extension of $\Z\left(\C(a^p-a)\right)$, and the later is infinite by Theorem~\ref{ZZZ}. By Fact~\ref{KSW}, one has $a\in \Z\left(\C(a^p-a)\right)$ and thus $\C(a^p-a)\subset \C(a)$. Now, assume for a contradiction that $b^{-1}ab=a+1$ holds. We deduce \[b^{-1}(a^p-a)b=(a+1)^p-(a+1)=a^p-a,\] a contradiction with the above claim.\end{proof}

\begin{corollaire}For every element $a$ in an \NIP{} division ring of characteristic $p$, one has \[\C(a^p)=\C(a).\]\end{corollaire}

\begin{proof}The element $a$ is algebraic over the field $\Z(\C(a^p))$. Since $ax-xa=1$ has no solution in $\C(a^p)$, by Fact~\ref{LamB}, $a$ is separable over $\Z(\C(a^p))$ so $a\in \Z(\C(a^p))$ and $\C(a^p)\subset \C(a)$.\end{proof}

\section{Linear preliminaries on difference division rings}\label{S:L}
Let $(\R,\sigma)$ be a division ring equipped with a ring morphism $\sigma$. We call the pair $(\R,\sigma)$ a \emph{difference division ring}, as in the commutative case \cite[p. 57]{Cohn1965}. We write $\FR$ for the division subring defined by $\sigma(x)=x$ and we make the additional assumptions:\begin{itemize}
\item that the dimension $[\R:\FR]_{\rm right}$ is infinite,
\item that $\sigma$ is surjective on $\R$.
\end{itemize}
In an attempt to make this paper self-contained, we gather in this Section the needed results from \cite{Milliet2018} concerning the structure of those subsets of $\R^n$ that are defined by linear equations involving~$\sigma$. We state them in all generality, although they will be (mainly) applied in the case where $\sigma=\sigma_a$ is a conjugation map by some transcendental element $a$ over $\Z(\R)$. 
\subsection{1-Twists}We define the set of \emph{$1$-twists} 
\[{\Rs}=\left\{\ \sum_{i=0}^n r_i{\sigma^i}\colon\bar r\in {\R}^{n+1},\ n\in\mathbf N\right\},\] a left $\R$-vector space with basis $\left\{\sigma^i\colon i\in\mathbf N\right\}$.
Equipped with the sum
\[\sum_{i=0}^n r_i{\sigma^i}+\sum_{j=0}^n s_j{\sigma^j}=\sum_{k=0}^n (r_k+s_k){\sigma^k}\] and the obvious composition law
\[\left(\sum_{i=0}^n r_i{\sigma^i}\right)\left(\sum_{j=0}^n s_j{\sigma^j}\right)=\sum_{i=0}^n\sum_{j=0}^n r_i\sigma^i(s_j){\sigma^{i+j}},\]
$\Rs$ is a unitary (we also write $\id$ for $\sigma^0$) associative domain. Generalising Ore's \mbox{\cite[Theorem 1]{Ore1933}} that the ring of $p$-polynomials $K[x^p]$ form a Euclidean domain when $K$ is a perfect field of characteristic $p$, the 
domain $\Rs$ is also Euclidean with the natural degree function, from which follows:

\begin{fact}[factorisation, \protect{\cite[Lemma 3.2]{Milliet2018}}]\label{3.2}Let $\rho$ be a $1$-twist of degree $n+1$ having a nonzero root~$a$. There is a $1$-twist $\delta$ of degree $n$ such that $\rho=\delta\left(\sigma-\sigma(a) a^{-1}\id\right)$.\end{fact}

Following \cite[p. 58]{Cohn1965}, we call a difference division ring $(E,\tau)$ such that $\R\subset E$ and $\tau\colon E\rightarrow E$ extends $\sigma\colon\R\rightarrow\R$, a \emph{difference extension of $(\R,\sigma)$}. By analogy with the definition in \mbox{\cite[p. 215]{ADH2017}} given for differential fields, although another terminology also exists for difference fields (see \emph{e.g.} \cite[Lemma 9.1 p. 17]{Scanlon2003} or \cite[Definition 4.3 p. 15]{Point2006}), we say that the difference division ring $(\R,\sigma)$ is \emph{linearly surjective} if for every nonzero $1$-twist~$\delta$, the equation $\delta(x)=1$ has a solution in $\R$.

\begin{fact}[\protect{\cite[Theorem 6.3]{Milliet2018}}]\label{6.3}Any $(\R,\sigma)$ has a linearly surjective difference extension.\end{fact}

\subsection{\texorpdfstring{$\sigma$}{}-Linear sets, \texorpdfstring{$\sigma$}{}-morphisms}Let $\Rn$ denote the left $\R$-vector space spanned by \[\left\{\sigma^{i_1}(x_1),\dots,\sigma^{i_n}(x_n)\colon (i_1,\dots,i_n)\in\mathbf N^n\right\}.\] $\Rn$ is a left $\Rs$-module. We call its elements \emph{$n$-twists}, and the zero set of a family $S$ of $n$-twists a \emph{$\sigma$-linear set}, which we write \[\V(S)=\{(x_1,\dots,x_n)\in \R^n\colon\delta(x_1,\dots,x_n)=0 \text{ for all }\delta\in S\}.\] A map between two $\sigma$-linear sets is a \emph{$\sigma$-morphism} if its coordinate maps are $n$-twists. A $\sigma$-morphism is a \emph{$\sigma$-isomorphism} if bijective and if its inverse is a $\sigma$-morphism.

\subsection{Zariski dimension}Given a subset $V\subset \R^n$, we write \[\I(V)=\{\delta\in\Rn\colon\delta(x_1,\dots,x_n)=0\text{ for all }(x_1,\dots,x_n)\in V\}.\] This is a $\Rs$-submodule of $\Rn$. We define the \emph{Zariski dimension of $V$} by \[\dim V=\dim_\Rs\Rn-\dim_\Rs\I(V),\] where $\dim_{\Rs}$ denotes the cardinal of any maximal $\Rs$-independent set (well-defined by \cite[Theorem 1.3]{Milliet2018} and \cite[Lemma 3.1]{Milliet2018}). For any submodule $I\subset\Rn$, we define its \emph{closure $\cl(I)$} by \[\cl(I)=\left\{\delta\in\Rn\colon\exists\gamma\in\Rs\setminus\{0\},\ \gamma\delta\in I\right\}.\] We say that a $\sigma$-linear set $V$ is \emph{radical} if $\cl(\I(V))=\I(V)$. Fact~\ref{5.4} below is \cite[Theorem 6.6]{Milliet2018}.

\begin{fact}\label{5.4}Given a $\sigma$-linear set $V$ and a twist $\delta$, one has $\dim \left(V\cap \V(\delta)\right)\geqslant \dim V-1.$\end{fact}

Fact~\ref{0} and Fact~\ref{00} are immediate consequences of \cite[Lemma 5.9]{Milliet2018}.

\begin{fact}\label{0}A $\sigma$-linear set $V$ has a unique radical component $V^0\subset V$ with $\dim V=\dim V^0$.\end{fact}

\begin{fact}\label{00}A radical $\sigma$-linear set of Zariski dimension $d$ is $\sigma$-isomorphic to $\R^{d}$.\end{fact}

\begin{fact}[\protect{\cite[Lemma 5.7]{Milliet2018}}]\label{5.7}Let $U$ and $V$ be $\sigma$-linear sets. Then  $\dim \left(U\times V\right)=\dim U+\dim V.$\end{fact} 

Fact~\ref{5.8} is a consequence of \cite[Theorem 5.8]{Milliet2018} and \cite[Theorem 6.4.2]{Milliet2018}.

\begin{fact}[Rank-Nullity]\label{5.8}Let $U$ be irreducible $\sigma$-linear and $f\colon U\rightarrow \R^n$ a $\sigma$-morphism. If $(\R,\sigma)$ is linearly surjective, then $\ima f$ is $\sigma$-linear, and $\dim U=\dim \ima f+\dim \ker f.$\end{fact}

\subsection{A particular radical group}Fact~\ref{6.7} bellow is inspired by \cite[Lemma 2.8]{KSW2011} and its improved version \cite[Lemme 5.3]{Hempel2016}. It plays a crucial role in \cite{KSW2011} and \cite{Hempel2016} in the particular case when the pair $(\R,\sigma)$ is an algebraically closed field $(K,{\rm Frob})$ of characteristic $p$ equipped with the Frobenius. In that particular case, if $\left\{b_1^{-1},\dots,b_n^{-1}\right\}$ are $\mathbf F_p$-linearly independent,  \cite[Lemme 5.3]{Hempel2016} states that, $G_{\bar b}$ is \emph{connected} as an algebraic group (\emph{i.e.} has no subgroup of finite index defined by polynomials), whereas Fact~\ref{6.7} only states that $G_{\bar b}$ has no subgroup of finite index defined by $p$-polynomials. But one recovers the conclusion of \cite[Lemme 5.3]{Hempel2016} knowing that $G_{\bar b}$ is $\sigma$-isomorphic to $(K,+)$ by Fact~\ref{00}, and $(K,+)$ is connected, so that $G_{\bar b}$ is connected as well. The proof of Fact~\ref{6.7} uses Fact~\ref{6.3} and \cite[Theorem 6.4]{Milliet2018} stating that $\sigma$-linear sets project onto $\sigma$-linear sets over a linearly surjective division ring.

\begin{fact}[see \protect{\cite[Lemma 6.7]{Milliet2018}}]\label{6.7}Given a natural number $n\geqslant1$ and $\bar b=(b_1,\dots,b_n)$ in $\R^\times$, we consider the $\sigma$-linear set defined by \[G_{\bar b}=\left\{(x_1,\dots,x_n)\in\R^n\colon b_1(\sigma x_1-x_1)=b_i(\sigma x_i-x_i)\text{ \rm for all }1\leqslant i\leqslant n\right\}.\] Then $G_{\bar b}$ is radical if and only if $\left\{b_1^{-1},\dots,b_n^{-1}\right\}$ are left $\FR$-linearly independent.\end{fact}

\section{A new look at the stable case}

\subsection{Stable division rings of prime characteristic}We begin by proposing an alternative proof of the stable case, that does not use the fact that iterates of $\sigma_a-\id$ are uniformly definable in characteristic~$p$ (where $\sigma_a$ is the conjugation map by~$a$). The part of the argument that mimics Scanlon's result \cite[Proposition 1]{Scanlon1999} has the advantage to be valid in any characteristic. We recall the definition of a stable structure. Given a natural number $k\in\mathbf N$ and a structure $(M,L)$, an $L$-formula $\phi(\bar x,\bar y)$ with $|\bar x|=|\bar y|=\ell$ has the \emph{$k$-order property} if there are $\ell$-tuples $\bar a_1,\dots,\bar a_{k-1}$ in $M$ such that for any $i,j<k$, \[\left(M\models\phi(\bar a_i,\bar a_j)\right)\iff i<j.\]

\begin{definition}[Shelah]A structure $(M,L)$ is \emph{stable} if for every $L$-formula $\phi(\bar x,\bar y)$, there is a natural number $k\in\mathbf N$ such that $\phi(\bar x,\bar y)$ does not have the $k$-order property.\end{definition}

The above is adapted from \cite[Definition 2.9]{Chernikov2015}. It is not the original definition \cite[Definition~2.2 p. 9]{Shelah1990}, but is equivalent to it by \cite[Theorem 2.13 p. 304]{Shelah1971} and by the Compactness Theorem. The following chain condition can be found in \cite[Proposition 1.4]{Poizat1987}. Note the similarity between Fact~\ref{BS} and Fact~\ref{SDCC}.

\begin{fact}[Stable descending chain condition]\label{SDCC}In a stable group, to any formula $\phi( x,\bar y)$ is associated a natural number $n\in\mathbf N$ such that the intersection of {\bf any} family $\{G_i\colon i\in I\}$ of subgroups defined respectively by the formulas $\{\phi(x,\bar a_i)\colon i\in I\}$ be the intersection of at most $n$ among them.\end{fact}

\begin{fact}[\protect{\cite[Theorem 2.1]{Milliet2011}}]\label{stable}A stable division ring of characteristic $p$ has finite dimension over its centre.\end{fact}

\begin{proof}It suffices to show that for every such division ring $\R$ and $a\in\R$, the dimension $[\R:\C(a)]$ is finite (by the stable descending chain condition~\ref{SDCC} applied to centralisers, this will imply that $\R$ has finite dimension over a commutative subfield, hence over its centre). Let us assume for a contradiction that $[\R:\C(a)]$ is infinite for some $a\in\R$. Let $\sigma_a$ be the conjugation map by $a$ and $\gamma=\sigma_a-\id$. We shall show that $\gamma\colon\R\rightarrow \R$ is onto, a contradiction with Corollary~\ref{metro}. We adapt the proof of \cite[Proposition 1]{Scanlon1999}. By the stable descending chain condition~\ref{SDCC}, there are a natural number $n\in\mathbf N$ and an $n$-tuple $\bar b=(b_1,\dots,b_n)$ of elements in $\R^\times$ such that \[I=\bigcap_{b\in\R^\times} b\cdot  \gamma(\R)=\bigcap_{b\in \bar b}b\cdot\gamma(\R).\] Let $G_{\bar b}$ the $\sigma_a$-linear set defined by \[G_{\bar b}=\left\{(x_1,\dots,x_n)\in \R^{n}\colon b_1\cdot\gamma(x_1)=b_i\cdot\gamma(x_i)\text{ for all } 1\leqslant i\leqslant n\right\}.\] This is an intersection of $n-1$ many $\sigma_a$-hypersurfaces of $\R^{n}$, so $\dim G_{(b_1,\dots,b_n)}\geqslant 1$ by Fact~\ref{5.4}. By Fact~\ref{0} and Fact~\ref{00}, the group $G_{\bar b}$ has infinite right $\FR$-dimension, so $I$ contains a nonzero element. Since $I$ is a left ideal of $\R$, one must have $I=\R$, hence $\gamma$ is onto, as desired.\end{proof}

\begin{remarque}\label{rk}Separably closed fields are currently the only known examples of infinite stable fields \cite[Theorem 3]{Wood1979}. From the conjecture \cite[p. 1]{Chatzidakis1997} \emph{every infinite stable field is separably closed}, follows \emph{every stable division ring of characteristic $p$ is a field}. For if $\R$ is a stable division ring of characteristic $p$ that is not a field, then~$\R$ has finite dimension over its centre by Fact~\ref{stable}. Pick some $a\in \R\setminus \Z(\R)$. By Corollary~\ref{metro}, the equation \mbox{$ax-xa=1$} has no solution, so that the extension $\Z(\R)(a)/\Z(\R)$ is  separable by Fact~\ref{LamB}. We do not know whether the reverse implication is true.
\end{remarque}

\begin{remarque}Bounded PAC fields are currently the only known examples of infinite simple fields \cite[Corollary 4.8]{CP1998}, and from the conjecture \emph{every infinite simple field is PAC}, follows \emph{every simple division ring of characteristic $p$ is a field}, since on the one hand, such a division ring must have finite dimension over its centre by \cite[Theorem 3.5]{Milliet2011}, and on the other hand its centre has a trivial Brauer group by \cite[Theorem 11.6.4]{FJ2008}. Also, since the iterated kernels of $\sigma_a-\id$ are uniformly definable in characteristic $p$, the map $\sigma_a-\id$ is not onto in an NSOP division ring of characteristic~$p$. In the proof of Fact~\ref{stable}, the stable chain condition is applied to uniformly definable vector spaces over an infinite division ring, so the argument remains valid for a simple division ring of characteristic~$p$, using the simple descending chain condition \cite[Theorem 4.2.12]{Wagner2000}.\end{remarque}

\subsection{Stable and simple difference fields}Let us point out consequences that concern stable or simple difference fields. It is noticed in \cite[Lemma 2.11]{CH2014} that any model of ACFA is linearly surjective. Recall that ACFA is supersimple \cite{CH1999} and that a supersimple difference field is inversive (follows from \cite{PP1995} or \cite[Fact 4.2.(ii)]{PSW1998}). With a proof similar as the one of Fact~\ref{stable}, and arguing as in the proof of \cite[Theorem 3.2]{KSW2011}, one can withdraw the uniform definability assumption in \cite[Proposition 3.6]{Milliet2011}:

\begin{theoreme}\label{T:SW}If $(K,\sigma)$ is a difference field with a simple theory and $k=\bigcap\sigma^n({K})$, then\begin{itemize}\item either $[K:\FR]$ is finite,\item or $\FR$ is finite, and the index $|K/\delta(K)|$ is finite for every $\delta\in k[\sigma]$ of valuation zero,\item or every $\delta\in k[\sigma]$ of valuation zero is surjective.\end{itemize}\end{theoreme}

The first case occurs when $K/F$ is Galois over a simple field $F$, and $\sigma$ a nontrivial element of ${\rm Gal}(K/F)$. The second case occurs \emph{e.g.} when $K$ is a pseudo-finite field of characteristic~$p$ with the Frobenius (by \cite[Lemma 4.5]{Hrushovski2002} or \cite[Proposition 4.5]{CDM1992}), in which case the index $|K/\delta(K)|$ is bounded by $|\ker\delta|$ by \L os Theorem, and maybe greater than one, \emph{e.g.} if $\delta$ is the Artin-Schreier map. The assumption on the valuation cannot be dropped as witnessed by an unperfect separably closed field. Since an infinite stable field has no proper definable additive subgroup of finite index, from Theorem~\ref{T:SW}, one recovers Scanlon's \cite[Proposition 1]{Scanlon1999}. Note that if $(K,\sigma)$ is inversive and $\FR$ infinite in Theorem~\ref{T:SW}, then $(K,\sigma)$ is linearly surjective.

\begin{proof}[Proof of Theorem~\ref{T:SW}]We may assume that $K$ is infinite, $\aleph_0$-saturated, that $[K:\FR]$ is infinite and that $\sigma$ is injective. By the Compactness Theorem and saturation hypothesis, there is a transcendental element $x$ over $\FR$. For all $n\in\mathbf N$, the element $\sigma^n(x)$ is also transcendental over $\FR$. By the Compactness Theorem, there is an element in $k$ that is transcendental over $\FR$, so the dimension $[k:k\cap\FR]$ is infinite. Let $\delta\in k[\sigma]$ be of valuation zero. We shall show that $\delta(K)$ has finite additive index in $K$. Let $\mathcal G=\left\{a\cdot\delta(K)\colon a\in k^{\times}\right\}$, a $k^\times$-invariant family. By \cite[Fact~3.1]{KSW2011}, there is a $k^\times$-invariant additive subgroup $N\leqslant K$ containing a finite intersection of groups in $\mathcal G$, say $\bigcap_{a\in\bar a}a\cdot\delta(K)$ for some finite $n$-tuple $\bar a=(a_1,\dots,a_n)$ of elements in $k^\times$, such that the additive index $|N/N\cap G|$ is finite for all $G\in\mathcal G$. Define the $\sigma$-linear group $G_{\bar a}(k)$ by \[G_{\bar a}(k)=\left\{(x_1,\dots,x_n)\in k^n\colon a_1\cdot\delta(x_1)=a_i\cdot\delta(x_i)\text{ \rm for all }a_i\in\bar a\right\}.\] The difference field $(k,\sigma)$ is inversive and $[k:k\cap\FR]$ infinite. As $G_{\bar a}(k)$ is the intersection of $n-1$ many $\sigma$-hypersurfaces, it has Zariski dimension at least $1$ by Fact~\ref{5.4}, and in fact precisely $1$ inductively on $n$ using Fact~\ref{5.8}. The group $G_{\bar a}^0(k)$ is $\sigma$-isomorphic to $(k,+)$ by Facts~\ref{0} and~\ref{00}. So $G_{\bar a}(k)$ has infinite $k\cap\FR$-dimension, and $\bigcap_{a\in\bar a}a\cdot\delta(k)$ is nonzero, so $N\cap k$ is nonzero as well. But $N\cap k$ is $k^\times$-invariant, hence an ideal of $k$, and must equal $k$, so that $k/k\cap\delta(K)$ embeds in $N/N\cap\delta(K)$. Now putting $\delta=a_n\sigma^n+\dots+a_0\id$ with $(a_0,\dots,a_n)$ in $k$ and $a_0\neq 0$, one has for all $a\in K$ the equality \[a=a_0^{-1}a_n\sigma^n(a)+\dots+a_0^{-1}a_1\sigma(a)+\delta(-a),\quad\text{ whence }\quad K=\sigma(K)+\delta(K).\] We show inductively on $n\in\mathbf N$ that  $K=\sigma^n(K)+\delta(K)$ holds for every $\delta\in k[\sigma]$ having valuation zero. If $K=\sigma^n(K)+\delta(K)$ holds for ever  such $\delta$, let $\delta=a_n\sigma^n+\dots+a_0\id$ in $k[\sigma]$ with $a_i=\sigma^i(b_i)$ and $b_i\in k$ (obtained by the Compactness Theorem). By induction hypothesis applied to $$\gamma=\sigma^{n-1}(b_n)\sigma^n+\dots+b_1\sigma+\sigma^{-1}(a_0)\id,$$ one has $K=\sigma^n(K)+\gamma(K)$, whence $\sigma(K)=\sigma^{n+1}(K)+\sigma\gamma(K)$. But $\sigma\gamma=\delta\sigma$, so \[K=\sigma(K)+\delta(K)=\sigma^{n+1}(K)+\delta\sigma(K)+\delta(K)=\sigma^{n+1}(K)+\delta(K).\] By the Compactness Theorem, one has $K=k+\delta(K)$,  and thus \[|K/\delta(K)|=|k/k\cap\delta(K)|\leqslant |N/N\cap\delta(K)|,\] and $|K/\delta(K)|$ is finite, as claimed.\end{proof}

It is shown in \cite[Proposition 3]{Hrushovski1989}, using \cite[Theorem 1]{CS1980}, that if $(K,\sigma)$ is a superstable difference field, then either $\sigma$ is trivial, or $\FR$ is finite. As a consequence of Theorem~\ref{T:SW}, one has:

\begin{corollaire}\label{C:SW}If $(K,\sigma)$ is a stable difference field of characteristic $p$, then either $[K:\FR]$ or $\FR$ is finite.
\end{corollaire}

\begin{proof}If both $[K:\FR]$ and $\FR$ are infinite, by Theorem~\ref{T:SW}, there is $x\in K$ such that $\sigma(x)-x=1$, from which follows $\sigma(x^p-x)=x^p-x$. But $\FR$ is Artin-Schreier closed, so $x\in\FR$, a contradiction.\end{proof}

The conclusion of Corollary~\ref{C:SW} fails for a simple field of characteristic $p$, as witnessed by ACFA. The analogous statement valid in all characteristics seems to be the following.

\begin{proposition}Let $(K,\sigma,\tau)$ be a stable field structure with commuting ring morphisms $\tau$ and $\sigma$, such that ${\rm Fix}(\tau)\subset\FR$. Then either $[K:\FR]$ or $[\FR:{\rm Fix}(\tau)]$ is finite.\end{proposition}

\begin{proof}Assume for a contradiction that both $[K:\FR]$ and $[\FR:{\rm Fix}(\tau)]$ are infinite. Then, by Theorem~\ref{T:SW} applied to $(K,\sigma)$, there is an $x$ such that $\sigma x-x=1$, from which follows $\sigma(\tau x-x)=\tau x-x$. So $\tau x-x$ belongs to $\FR$; but $(\FR,\tau)$ is a stable difference field with infinite $[\FR:{\rm Fix}(\tau)]$, so by Theorem~\ref{T:SW}, there is $a\in\FR$ with $\tau x-x=\tau a-a$. Putting $y=x-a$, one has $\sigma y-y=1$, and also $\tau y=y$, and so $\sigma y=y$ by assumption, a contradiction.\end{proof}

\section{The NIP case}

\begin{theoreme}\label{gegen}An \NIP{} division ring of characteristic $p$ has finite dimension over its centre.\end{theoreme}

\begin{proof}It suffices to show that for every such division ring $\R$ and $a\in\R$, the dimension \mbox{$[\R:\C(a)]$} is finite (for in that case, the set $\left\{[\R:\C(a)]\colon a\in\R\right\}$ is bounded by the Compactness Theorem, hence any descending chain of centralisers must stabilise by the \NIP{} chain condition~\ref{BS}). Let us assume for a contradiction that $[\R:\C(a)]$ is infinite for some $a\in\R$. Let $\sigma_a$ be the conjugation map by~$a$ and $\gamma=\sigma_a-\id$. We shall show that $\gamma\colon\R\rightarrow\R$ is onto, a contradiction with Corollary~\ref{metro}. We adapt the proof of \cite[Theorem 4.3]{KSW2011}. For every natural number $m\geqslant 1$ and infinite tuple $\bar b\in\R^{\mathbf N}$, let us consider the $\sigma_a$-linear  additive subgroup $G_{(b_1,\dots,b_m)}$ of~$\R^m$ defined by \[G_{(b_1,\dots,b_m)}=\left\{(x_1,\dots,x_m)\in \R^{m}\colon b_1\cdot\gamma(x_1)=b_i\cdot\gamma(x_i)\text{ for all } 1\leqslant i\leqslant m\right\}.\] One has $\dim G_{(b_1,\dots,b_m)}\geqslant 1$ by Fact~\ref{5.4}. Consider the first projection $\pi_1\colon G_{\bar b}^{m}\rightarrow \R$. One has \[\ker\pi_1=\{0\}\times\ker\gamma\times\dots\times \ker\gamma.\] Since $\ker\gamma=\C(a)$ is defined by the equation $\sigma-\id$, the $\Rs$-module $\I(\ker\gamma)$ has $\Rs$-dimension~$1$, so $\ker\gamma$ has Zariski dimension $0$. By Fact~\ref{5.7}, the kernel $\ker\pi_1$ also has Zariski dimension $0$. By Fact~\ref{5.8}, one has \[\dim G_{(b_1,\dots,b_m)}=1.\] Since $[\R:\C(a)]$ is infinite, by Fact~\ref{6.7}, one can chose an infinite tuple $\bar b\in\R^{\mathbf N}$ such that the $\sigma_a$-linear group $G_{(b_1,\dots,b_m)}$ is radical for every $m$. By the \NIP{} chain condition~\ref{BS}, there are natural numbers~$n$ and $i$ such that \[\bigcap_{j\in\{1,\dots,n+1\}} b_j\cdot\gamma(\R)=\bigcap_{j\in\{1,\dots,n+1\}\setminus\{i\}} b_j\cdot\gamma(\R).\] Note that this is the only place where we use the \NIP{} hypothesis to show that $\gamma$ is onto. We may reorder the tuple $(b_1,\dots,b_{n+1})$ if need be and assume that $i=n+1$, so that the projection \[\pi\colon G_{(b_1,\dots,b_{n+1})}\rightarrow G_{(b_1,\dots,b_{n})}\] on the $n$ first coordinates is onto. Since $G_{(b_1,\dots,b_{n+1})}$ and $G_{(b_1,\dots,b_{n})}$ are radical, by Fact~\ref{00}, there are two $\sigma_a$-isomorphisms \[\alpha\colon G_{(b_1,\dots,b_{n+1})}\rightarrow (\R,+)\quad\text{ and }\quad\beta\colon G_{(b_1,\dots,b_{n})}\rightarrow (\R,+).\] The $\sigma_a$-morphism $\rho=\beta\pi\alpha^{-1}$ makes the following diagram commute.
\begin{center}\begin{tikzpicture}
  \matrix (m) [matrix of math nodes,row sep=3em,column sep=4em,minimum width=2em]
  {
    G_{(b_1,\dots,b_{n+1})} & G_{(b_1,\dots,b_{n})} \\
     (\R,+) & (\R,+) \\};
  \path[-stealth]
    (m-1-1) edge node [left] {$\alpha$} (m-2-1)
            edge node [above] {$\pi$} (m-1-2)
    (m-2-1.east|-m-2-2) edge node [below] {$\rho$}
             (m-2-2)
    (m-1-2) edge node [right] {$\beta$} (m-2-2)
            ;
\end{tikzpicture}\end{center}
Since $\ker\pi=\{0\}\times\dots\times\{0\}\times\C(a)$, the map $\rho$ has a nontrivial kernel. Let $c\in\ker\rho\setminus\{0\}$, and put $\rho^*(x)=\rho(cx)$. Then $\rho^*\colon\R\rightarrow\R$ is onto since $\rho$ is. One also has $\ker \rho^*=c^{-1}\cdot\alpha(\ker\pi)$.
Since $\ker\pi$ has right $\C(a)$-dimension $1$, $\ker\rho^*$ also has right $\C(a)$-dimension $1$. Since $\ker\rho^*$ contains $1$, one must have $\ker\rho^*=\ker\gamma=\C(a)$, and the equality $\ker\rho^*=\ker\gamma$ holds in any difference extension of~$(\R,\sigma_a)$. By Fact~\ref{3.2}, the twist $\rho^*$ factorises in $\rho^*=\delta \gamma$ with $\gamma=\sigma_a-\id$. If $x\in\ker\delta$, then $x=\gamma(y)$ for some $y$ in a linearly surjective extension of $(\R,\sigma_a)$ given by Fact~\ref{6.3}, hence $y\in \ker\rho^*=\ker \gamma$ so $x=0$. It follows that $\delta$ is bijective, so $\gamma$ is onto $\R$, which is the desired contradiction.\end{proof}

Elb\'ee \cite{Elbee2016} has a few lines proof, using computational properties of the dp-rank, that a \emph{strongly} \NIP{} division ring of dp-rank $n$ has dimension at most $n$ over its centre (in any characteristic), in the same vein as the proof of \cite[Proposition 7.9]{Hrushovski2002} for division rings of finite $S_1$-rank.

\begin{proposition}If $(K,\sigma)$ is an \NIP{} difference field of characteristic $p$, then either $[K:\FR]$ or $\FR$ is finite.\end{proposition}

\begin{proof}If $[K:\FR]$ is infinite, an argument as in the proof of Theorem~\ref{gegen} (working over $k=\bigcap\sigma^n(K)$) shows that $\sigma-\id$ is surjective, so the argument in the proof of Corollary~\ref{C:SW} applies.\end{proof}

\section{The \texorpdfstring{$\rm NTP_2$}{} case}

Shelah had introduced already in \cite[Theorem 0.2]{Shelah1980} a large class of structures now called \NTP{} (for ``not the tree property of the second kind'') including both \NIP{} and simple structures. We note here that the conclusions of \cite[Theorem 3.5]{Milliet2011} and Theorem~\ref{gegen} extend to \NTP{} division rings of characteristic $p$, using Chernikov, Kaplan and Simon's results\begin{itemize}
\item that \NTP{} groups satisfy a descending condition chain \cite[Theorem 2.4]{CKS2015}:

\begin{fact}[\NTP{} descending chain condition]\label{F:N} In an \NTP{} group, to any formula $\phi(x,\bar y)$ is associated a natural number $n\in\mathbf N$ such that the intersection of any finite family $\{G_i\colon i<k\}$ of {\bf normal} subgroups defined respectively by the formulas $\left\{\phi(x,\bar a_i)\colon i<k\right\}$ has finite index in a subintersection of at most $n$ among them.\end{fact}

\item that \NTP{} fields have finitely many Artin-Schreier extensions \cite[Theorem 3.1]{CKS2015}, conclusion which was remarked by Wagner for simple fields \cite[Theorem 3.2]{KSW2011}.\end{itemize}

We recall the definition of an \NTP{} structure (although we shall not use it directly). Given a natural number $k\in\mathbf N$, and a structure $(M,L)$, an $L$-formula $\phi(x,\bar y)$ has the \emph{$k$-tree property$_2$} if there is an array $(\bar a_{i,j})_{i,j<k}$ of tuples in $M$ such that $\{\phi(x,a_{i,j})\colon j<k\}$ are pairwise inconsistent for each $i<k$ and $\{\phi(x,a_{i,f(i)})\colon i<k\}$ is consistent for any $f\colon\{1,\dots,k-1\}\rightarrow\{1,\dots,k-1\}$.

\begin{definition}[Shelah] A structure $(M,L)$ is \NTP{} if for every $L$-formula, there is a natural number $k\in\mathbf N$ such that $\varphi(x,\bar y)$ does not have the $k$-tree property$_2$.\end{definition}

\begin{lemme}\label{L:N}An \NTP{} division ring of characteristic $p$ has a definable division subring of finite codimension in which $xy-yx\neq1$ holds.\end{lemme}

\begin{proof}We may assume that the ambient division ring is $\aleph_0$-saturated, and first claim:\begin{Aff}\label{C:N}For any element $a$ of infinite order, one has $\C(a^p-a)\subset \C\left(a^{p^n}\right)$ for some $n\in\mathbf N$.\end{Aff}
By \cite[Theorem 3.1]{CKS2015}, the field $\Z(\C(a^p-a))$ has finitely many Artin-Schreier extensions. Since $\Z(\C(a^p-a))(a^{p^i})$ is an Artin-Schreier extension for each $i\in\mathbf N$, and since the set $\{a^{p^i}\colon i\in\mathbf N\}$ is infinite, there is an $n\in\mathbf N$ such that $a^{p^n}\in\Z(\C(a^p-a))$, whence $\C(a^p-a)\subset\C(a^{p^n})$, as claimed. Writing $M$ for the set defined by $\exists b \left(b^{-1}xb-x=1\right)$, we split the proof of Lemma~\ref{L:N} into two cases:\begin{case}{$\bf1.$ \emph{There exists $\bar\alpha=(\alpha_1,\dots,\alpha_q)$, a finite tuple of elements of finite order, such that $M\cap\C(\bar\alpha)$ contains only finitely many $\bar\omega=(\omega_1,\dots,\omega_r)$ of finite order.}} We consider the division subring $C=\C(\bar\alpha,\bar\omega)$. It has finite codimension by Fact~\ref{F:B}, and we claim that it satisfies $xy-yx\neq 1$. Assume for a contradiction that $b^{-1}ab-a=1$ holds for some $(a,b)$ in $C$. If $a$ has finite order, then it is one of the $\omega_i$, so $a$ commutes with $b$, a contradiction. So $a$ has infinite order. One has $$b^{-1}(a^p-a)b=(a+1)^p-(a+1)=a^p-a,$$ and thus $b\in\C(a^{p^n})$ for some $n\in\mathbf N$. But one also has $b^{-1}\left(a^{p^n}\right){b}=a^{p^n}+1$, a contradiction.\end{case}
\begin{case}$\bf 2.$ \emph{For all $\bar\alpha=(\alpha_1,\dots,\alpha_q)$ of finite order, there are infinitely many elements $\{\omega_i\colon i\in I(\bar\alpha)\}$ of finite order in $M\cap\C(\bar\alpha)$.} 
 If the elements in $\{\omega_i\colon i\in I(1)\}$ have unbounded order, since $M$ is definable, by the Compactness Theorem and the saturation assumption, $M$ contains an element of infinite order, which leads to a contradiction as in Case~1, using Claim~\ref{C:N}. So the elements in $\{\omega_i\colon i\in I(1)\}$ have bounded order. They are all roots of a common polynomial $x^{p^n}-x$ for some $n\in\mathbf N$. Since the additive index $\left|C_1/C_2\right|$ is either $1$ or $\infty$ whenever $C_1$ is an infinite division ring with division subring $C_2$, by the \NTP{} descending chain condition~\ref{F:N} and Fact~\ref{F:B}, there are a natural number $k\in\mathbf N$ and a finite tuple $(\omega_1,\dots,\omega_{k})$ of elements of finite order such that $$\C(\omega_1,\dots,\omega_k)=\C(\omega_1,\dots,\omega_{k},\omega)$$ for any $\omega\in M$ of finite order. Putting $\bar\omega=(\omega_1,\dots,\omega_k)$, it follows that the elements in $\{\omega_i:i\in I(\bar\omega)\}$ form an infinite (by assumption) commuting set, and are zeros of $x^{p^n}-x$, a contradiction.\qedhere\end{case}\end{proof}

\begin{theoreme}\label{T:N}An \NTP{} division ring of characteristic $p$ has finite dimension over its centre.\end{theoreme}

\begin{proof}By Lemma~\ref{L:N}, it suffices to show that for every such division ring $\R$ satisfying $xy-yx\neq 1$, the dimension \mbox{$[\R:\C(a)]$} is finite for every $a$. Let us assume for a contradiction that $[\R:\C(a)]$ is infinite. As in the proof of Theorem~\ref{gegen}, we consider $\sigma_a$ the conjugation map by $a$, we write $\gamma=\sigma_a-\id$ and show that $\gamma$ is onto $\R$. This can be done choosing an infinite tuple $\bar b\in \R^{\mathbf N}$ such that the groups $G_{(b_1,\dots,b_n)}$ are radical for each $n\geqslant1$ thanks to Fact~\ref{6.7}, and applying the \NTP{} chain condition~\ref{F:N} to the family $\left\{b_i\cdot\gamma(\R)\colon i\in\mathbf N\right\}$, of right vector spaces over the infinite division ring $\C(a)$. This contradicts the assumption that $xy-yx\neq1$ holds in $\R$.\end{proof}

\begin{corollaire}The centre of an infinite \NTP{} division ring is infinite.\end{corollaire}

\bibliographystyle{plain}
\bibliography{mabib}

\end{document}